\documentclass{article}
\usepackage{graphicx} 
\usepackage[margin=1in]{geometry}
\usepackage{amsmath}
\usepackage{amsthm}
\usepackage{amssymb}
\usepackage{bbm}
\usepackage{tikz-cd}
\usepackage[noadjust]{cite}
\usepackage{hyperref}
\usepackage{authblk}
\usepackage{tikz}
\usepackage{mathdots}
\usepackage{enumerate}
\usepackage{comment}
\usepackage{stmaryrd}
\usepackage[normalem]{ulem}
\usepackage[toc,page]{appendix}
\usepackage{rotating}
\usepackage{subcaption}
\usepackage{slashbox}
\usepackage{multirow}
\usepackage{placeins}
\tikzset{
  symbol/.style={
    draw=none,
    every to/.append style={
      edge node={node [sloped, allow upside down, auto=false]{$#1$}}}
  }
}

\newtheorem{thm}{Theorem}[section]
\newtheorem{prop}[thm]{Proposition}
\newtheorem{lem}[thm]{Lemma}
\newtheorem{cor}[thm]{Corollary}

\newtheorem{thmalpha}{Theorem}
\newtheorem{coralpha}[thmalpha]{Corollary}

\newcommand{\K}{\mathcal{K}}
\newcommand{\Ll}{\mathcal{L}}
\newcommand{\M}{\mathcal{M}}
\newcommand{\N}{\mathcal{N}}

\newcommand{\googlebooks}[1]{(preview at \href{https://books.google.com/books?id=#1}{google books})}

\newcommand{\numdam}[1]{}

\theoremstyle{remark}

\theoremstyle{definition}
\newtheorem{ex}[thm]{Example}

\newtheorem{defn}[thm]{Definition}

\newtheorem{rmk}[thm]{Remark}

\DeclareMathOperator{\tr}{tr}

\DeclareMathOperator{\End}{End}

\DeclareMathOperator{\id}{id}
\DeclareMathOperator{\Irr}{Irr}
\DeclareMathOperator{\RHS}{RHS}
\DeclareMathOperator{\LHS}{LHS}
\DeclareMathOperator{\re}{Re}
\DeclareMathOperator{\im}{Im}

\allowdisplaybreaks

\title{Noncommutative Minkowski integral inequality and \\
a unitary categorification criterion for fusion rings}
\author{Junhwi Lim}
\date{}

\AtEndDocument{%
  \par
  \noindent
  \medskip
  \begin{tabular}{@{}l@{}}%
    \textsc{Department of Mathematics, Vanderbilt University, Nashville, TN 37240, USA}\\
    \textit{E-mail address}: \texttt{junhwi.lim@vanderbilt.edu}
  \end{tabular}
  }
\begin{document}

\maketitle
\begin{abstract}
    We prove a noncommutative analogue of Minkowski's integral inequality for commuting squares of tracial von Neumann algebras. The inequality implies a necessary condition for a quadruple of graphs to be realized as inclusion graphs of a commuting square of multi-matrix algebras. As a corollary, we obtain a unitary categorification criterion for based rings, in particular, fusion rings. 
\end{abstract}
\section{Introduction}
A central question in the construction of fusion categories is whether a fusion ring $\mathcal{R}$ is the Grothendieck ring of some fusion category. If such a category exists, then it is called a \textit{categorification} of $\mathcal{R}$.
A particularly interesting case is when the fusion category is unitary, as it encodes generalized symmetries in operator algebras and quantum physics.  
The construction of unitary fusion categories is closely related to the construction of subfactors. Indeed, using Popa's reconstruction theorem \cite{MR1055708,MR1278111}, we can construct a subfactor from a category and its Q-system \cite{MR1966524}. Another related question in subfactor theory is whether a quadruple of graphs constitutes inclusion graphs of a nondegenerate commuting square of multi-matrix algebras. If such a commuting square exists, then we can construct a hyperfinite $\rm{II}_1$ subfactor via the basic construction \cite{MR1473221}.

The obstruction to categorification of a fusion ring is the \textit{pentagon equation}, or the \textit{$6j$-symbols}. The construction of a unitary associator satisfying the pentagon equation involves solving a system of quadratic equations in a large number of variables and their complex conjugates. The construction of a commuting square from graphs also involves solving similar equations, called the \textit{biunitary condition} \cite{MR1473221}. However, there is no known systematic way to solve them, and in many cases there is no solution.

Liu, Palcoux and Wu proved a remarkable necessary condition for a fusion ring to admit a unitary categorification \cite{MR4292959}. This condition can be used as an efficient criterion for excluding fusion rings that do not admit a unitary categorification without having to verify the pentagon equation. Their criterion uses the quantum double construction \cite{MR1966525} and the positivity of convolution for unitary fusion categories. This criterion was later reformulated in terms of operator positivity and complete positivity of comultiplication and was called the \textit{primary criterion} in \cite{MR4686667}. Surprisingly, among 28,541 fusion rings in \cite{MR4644689}, 19,738 of them (about 68.37\%) do not pass this criterion. Etingof, Nikshych and Ostrik gave an alternative proof of the primary criterion using formal codegree and the Drinfeld center of the category \cite{MR4874555}. They also relate it to the Isaacs property on fusion categories.

Huang, Liu, Palcoux and Wu established variants of the primary criteria called \textit{localized} criteria and \textit{reduced twisted} criteria in \cite{MR4686667}. These criteria only use partial information on fusion rules. Thus, they can exclude multiple fusion rings with similar fusion rules at once and use much smaller matrices than the primary criteria.

While the criteria of Huang, Liu, Palcoux, and Wu are remarkably powerful, it is not known whether they can be generalized to the unitary categorification of \textit{based rings}, which is an infinite generalization of fusion rings \cite[Question 7.1]{MR4686667}. 
In this paper, we introduce a condition that serves as a unitary categorification criterion for based rings and more generally, a criterion for examining whether graphs admit solutions to the biunitary condition. 
Our criterion uses a noncommutative analogue of the \textit{Minkowski integral inequality}:
\begin{thmalpha}\label{NCMinA}
    Let $(\M,\tr)$ be a tracial von Neumann algebra,  $\N\subset \Ll,\K\subset \M$ be von Neumann subalgebras and $x\in \M$ be a positive element satisfying the following:
    \begin{enumerate}
        \item Either $\Ll$ and $\K$ commute, or $x$ is in the center of $\M$;\label{variants}
        \item The quadrilateral inclusion $$
    \begin{tikzcd}
        \Ll\arrow[r,symbol=\subset, "E_\Ll" {yshift=3pt}] & \M\\
        \N\arrow[r,symbol=\subset,"E_\N"'{yshift=-3pt}] \arrow[u,symbol=\subset,"E_\N"{xshift=-3pt}] & \K\arrow[u,symbol=\subset, "E_\K"' {xshift=3pt}]
    \end{tikzcd}
    $$
    is a commuting square where $E_\Ll$, $E_\K$, and $E_\N$ are the unique $\tr$-preserving conditional expectations.
    \end{enumerate}
    Then for $1\le p<\infty$ we have
    \begin{equation*}
        E_\Ll(E_{\K}(x)^p)^{1/p}\le E_{\K}(E_\Ll(x^p)^{1/p}).
    \end{equation*}
\end{thmalpha}
In Theorem \ref{NCMinA}, the conditional expectations replace the partial integrals in the original Minkowski integral inequality. The proof of the original inequality uses the Fubini theorem and then the H\"older inequality to one of the partial integrals (see for instance, \cite[Proposition 1.3]{MR2768550}). We assume the commuting square condition in Theorem \ref{NCMinA}, as it plays the role of a `noncommutative Fubini theorem.'
However, the H\"older inequality for partial integral cannot be replaced with the well-known \textit{noncommutative H\"older inequality} \cite[Theorem 2.1.4]{NCLp}, as it does not take into account the operator-valued conditional expectation. Thus, we establish a suitable generalization of the H\"older inequality in Lemma \ref{Holderlike}.

As a consequence of Theorem \ref{NCMinA}, we obtain the following result for commuting squares:
\begin{thmalpha}\label{commuting squareA}
    Let 
    \begin{equation*}
        \begin{tikzcd}
            A_{01}\arrow[r,symbol=\subset,"T_{01}^{11}"{yshift = 5pt}]&A_{11}\\
            A_{00}\arrow[r,symbol=\subset,"T_{00}^{10}"'{yshift = -5pt}] \arrow[u,symbol=\subset,"T_{00}^{01}" {xshift = -5pt}]&A_{10}\arrow[u,symbol=\subset,"T_{10}^{11}"'{xshift = 5pt}]
        \end{tikzcd}
    \end{equation*}
    be a commuting square of multi-matrix algebras. Here $T_{ij}^{kl}$ is the inclusion matrix for $A_{ij}\subset A_{kl}$.
    Let $\tr:A_{11}\rightarrow \mathbb{C}$ be a faithful trace on $A_{11}$ with trace vector $v^{11}=(v_k^{11})_{k=1}^{n_{11}}$ and let $v^{ij}=(v_{k}^{ij})_{k=1}^{n_{ij}}$ be the induced trace vector of $A_{ij}$. Then for all $1\le p<\infty$, $a_l\ge 0$ for $1\le l\le n_{11}$ and $1\le i\le n_{00}$, we have
    \begin{equation*}
        \left(\sum_{j=1}^{n_{10}}\left(\sum_{l=1}^{n_{11}} \frac{(T_{10}^{11})_{jl}\ a_lv_l^{11}}{v_j^{10}}\right)^p \frac{(T_{00}^{10})_{ij}\ v_j^{10}}{v_i^{00}}\right)^{\frac{1}{p}} \le \sum_{k=1}^{n_{01}}\left(\sum_{l=1}^{n_{11}}  \frac{(T_{01}^{11})_{kl}\ a_l^pv_l^{11}}{v_k^{01}}\right)^{\frac{1}{p}}\frac{(T_{00}^{01})_{ik}\ v_k^{01}}{v_i^{00}}.
    \end{equation*}
\end{thmalpha}
In the commuting square construction, this theorem can be used to exclude graphs that do not yield a commuting square, without checking the biunitary condition. In particular, endomorphism spaces of tensor products of objects in a spherical unitary tensor category form commuting squares, and any rigid unitary tensor category has a spherical pivotal structure \cite[Lemma 3.9]{MR2091457}. Therefore, Theorem \ref{commuting squareA} specializes to a unitary categorification criterion for based rings with given quantum dimension data:
\begin{coralpha}\label{fusion ringA}
    Let $\mathcal{C}$ be a spherical unitary (multi-)tensor category with $s$ isomorphism classes of simple objects where $s\in \mathbb{N}\cup\{\infty\}$. Suppose that the simple objects $c_i\in \Irr(\mathcal{C})$ satisfy the fusion rules $c_i\otimes c_j\cong \bigoplus_{l=1}^s N_{ij}^l c_l$ and let $d_i$ be the quantum dimension of $c_i$. Then for any triple of simple objects $c_i,c_j,c_k\in\Irr(\mathcal{C})$, $1\le p<\infty$, and $a_n\ge0$ for $1\le n\le s$, we have
    \begin{equation*}
        \left(\sum_{l=1}^{s}\left(\sum_{n=1}^{s} \frac{N_{lk}^n\ a_nd_n}{d_l}\right)^p \frac{N_{ij}^ld_l}{d_j}\right)^{\frac{1}{p}} \le \sum_{m=1}^{s}\left(\sum_{n=1}^{s}  \frac{N_{im}^n\ a_n^pd_n}{d_m}\right)^{\frac{1}{p}}\frac{N_{jk}^m d_m}{d_j}.
    \end{equation*}
\end{coralpha}
We tested the above inequality on the aforementioned dataset of fusion rings from \cite{MR4644689}. Out of 28,450 fusion rings, 5,793 were excluded, corresponding to 20.36\%. Among them, 540 items, that is, 1.898\%, cannot be excluded by the primary $3$-criterion of \cite{MR4686667}. Moreover, $23$ of these cannot be excluded by the \textit{$d$-number}, \textit{extended cyclotomic}, \textit{Lagrange}, and \textit{pseudo-unitarity Drinfeld} criteria in \cite{MR4434141}. Some of these cannot even be excluded by the criterion in \cite[Remark 3.2]{ABDP25}. 

In Section \ref{Sec_NCMin}, we prove a generalized H\"older inequality and Theorem \ref{NCMinA}. In Section \ref{Sec_App}, we show Theorem \ref{commuting squareA} and deduce Corollary \ref{fusion ringA}. Then we apply Corollary \ref{fusion ringA} to exclude fusion rings without unitary categorification. 

\paragraph{Acknowledgments.} The author was supported by U.S. ARO Grant W911NF-23-1-0026. The author would like to thank Dietmar Bisch and Zhengwei Liu for fruitful discussions and comments, and S\'ebastien Palcoux for generously providing his fusion ring datasets and comments. The author also thanks Ricardo Correa da Silva and Hun Hee Lee for helpful discussions and materials related to noncommutative $L^p$-spaces.

\section{Noncommutative Minkowski's integral inequality}\label{Sec_NCMin}

Throughout, every von Neuman algebra is \textit{tracial}, i.e. it admits a faithful normal \textit{tracial state} $\tr$ (or simply, a \textit{trace}). The set of positive elements of a von Neumann algebra $\M$ is denoted by $\M_+$ and the center of $\M$ by $Z(\M)$. Every inclusion $\N\subset \M$ of von Neumann algebras is \textit{unital}, i.e. $1_\N=1_\M$. For an element $x\in \M$ its spectrum is denoted by $\sigma(x)$. 
For $1\le p\le \infty$, let $1\le p'\le \infty$ satisfies $\frac{1}{p}+\frac{1}{p'}=1$. For $1\le p<\infty$ the \textit{$p$-norm} $\|x\|_p$ of $x\in \M$ is defined by $\|x\|_p=\tr(|x|^{p})^{1/p}$ for $1\le p<\infty$. The operator norm of $x$ is denoted by $\|x\|=\|x\|_\infty$.

\subsection{Reduced H\"older inequality}
In this section, we generalize the following noncommutative H\"older inequality:
\begin{thm}[{\cite[Theorem 2.1.4]{NCLp}}]
    Let $x,y\in \M_+$ and $1\le p\le \infty$. Then 
    $$
    \tr(xy)\le \|x\|_p\|y\|_{p'}.
    $$
\end{thm}

\begin{defn}
    A positive element $x\in \M_+$ is said to have a \textit{spectral gap} if there is an $\epsilon>0$ such that $(0,\epsilon)\cap \sigma(x)=\varnothing$.
\end{defn}

Let $g:[0,\infty]\rightarrow\mathbb{R}$ be defined as 
$$
g(t)=\left\{\begin{array}{ll}
     \log t,&  t>0,\\
     0,& t=0
\end{array}\right..
$$
Then for any positive element $x\in \M$ with a spectral gap, $g|_{\sigma(x)}$ is continuous. For a complex number $w\in\mathbb{C}$, we define $x^w$ by
$$
x^w=\sum_{n=0}^\infty\frac{1}{n!}(wg(x))^n 
$$
with the convention $w^0=1$ for any $w\in \mathbb{C}$, and $g(x)^0=x^0$ is the support projection of $x$. 
\begin{rmk}
    \begin{enumerate}
        \item It is immediate from the definition that
        $$
        (x^w)^*=x^{\overline{w}}.
        $$
        \item The map $w\mapsto x^w$ has derivative $w\mapsto g(x)x^w$. Indeed, for $0<|h|<1$
        \begin{align*}
            &\left\|\frac{x^{w+h}-x^{w}}{h}-g(x)x^w\right\|\le \| x^w\|\left\|\frac{x^{h}-x^{0}}{h}-g(x)\right\|=\| x^w\|\left\|\sum_{n=2}^\infty\frac{h^{n-1}}{n!}g(x)^n\right\|\\
            &\le \| x^w\|\sum_{n=2}^\infty\frac{|h^{n-1}|}{n!}\|g(x)\|^n\le \| x^w\|\sum_{n=2}^\infty\frac{|h|}{n!}\|g(x)\|^n\le\| x^w\|\exp(\|g(x)\|)|h|\ \xrightarrow{h \to 0}\  0.
        \end{align*}
    \end{enumerate}
\end{rmk}

\begin{lem}[Reduced H\"older inequality]\label{Holderlike}
    Let $x,y\in \M_+$ and $q\in\M$ be a projection. For $1< p< \infty$
    \begin{equation}
        \tr(qxq y)\le \tr(qx^pq)^{1/p}\tr(qy^{p'}q)^{1/p'}.\label{Holderlikeineq}
    \end{equation}
    If $x$ and $y$ are nonzero and have a spectral gap, then equality holds if and only if $[q,x]=0=[q,y]$ and $(xq)^p=\lambda (yq)^{p'}$ for some $\lambda\in (0,\infty)$.
\end{lem}
\begin{proof}
    We first consider the case when $x$ and $y$ have a spectral gap.
    Let $\Omega=\{w\in\mathbb{C}: 0\le \re(w)\le 1\}$,  and let $f:\Omega\rightarrow\mathbb{C}$ be defined by
    $$
    f(w)=\tr(qx^{pw}qy^{p'(1-w)}).
    $$
    Then $f$ is holomorphic on the interior of $\Omega$ by differentiability of $x^{pw}$ and $y^{p'(1-w)}$ and norm-continuity of $\tr$. Note that 
    \begin{align}
        &|f(u+iv)|=\left|\tr\left(\left[y^{p'\frac{1-u}{2}}qx^{p\left(\frac{u}{2}+iv\right)}\right]\left[x^{p\frac{u}{2}}qy^{p'\left(\frac{1-u}{2}-iv\right)}\right]\right)\right|\nonumber\\
        &\le \tr\left(y^{p'\frac{1-u}{2}}qx^{p\left(\frac{u}{2}+iv\right)}x^{p\left(\frac{u}{2}-iv\right)}qy^{p'\frac{1-u}{2}}\right)^{1/2}
        \tr\left(y^{p'\left(\frac{1-u}{2}+iv\right)}qx^{p\frac{u}{2}}x^{p\frac{u}{2}}qy^{p'\left(\frac{1-u}{2}-iv\right)}\right)^{1/2}\label{cs}\\
        &=\tr(qx^{pu}qy^{p'(1-u)})^{1/2}
        \tr(qx^{pu}qy^{p'(1-u)})^{1/2}=\tr(qx^{pu}qy^{p'(1-u)})=f(u).\label{traciality}
    \end{align}
    where we used the Cauchy-Schwarz inequality in \eqref{cs} and traciality in \eqref{traciality}. This implies that $f(w)$ is bounded. Moreover, from the positivity of the trace $\tr$, we obtain
    \begin{align}
        |f(1+iv)|\le& \tr(qx^{p}qy^{0}) \le \tr(qx^{p}q),&
        |f(iv)|\le& \tr(qx^{0}qy^{p'}) \le \tr(qy^{p'}q).\label{ineq3}
    \end{align}
    By the three lines theorem,
    \begin{align}
        |\tr(qx^{pw}qy^{p'(1-w)})|=|f(w)|\le \left|\left(\sup_{v\in\mathbb{R}}|f(1+iv)|\right)^{w}\left(\sup_{v\in\mathbb{R}}|f(iv)|\right)^{1-w}\right|=\left|f(1)^wf(0)^{1-w}\right|\label{three lines0}
    \end{align}
    where we used \eqref{traciality} in the last identity.
    In particular,
    \begin{align}
        \tr(qxqy)=f(1/p)\le f(1)^{1/p}f(0)^{1/p'}.\label{three lines}
    \end{align}
    Combining \eqref{ineq3} and \eqref{three lines} yields \eqref{Holderlikeineq}.

    Next, we consider the general case where $x$ and $y$ do not necessarily have a spectral gap. Let $x_n=x\mathbbm{1}_{\left[\frac{1}{n},\|x\|\right]}(x)$ and $y_m=y\mathbbm{1}_{\left[\frac{1}{m},\|y\|\right]}(y)$. Then $x_n\xrightarrow{n\to\infty} x$ and $y_m\xrightarrow{m\to\infty}y$ in the operator norm. Thus, 
    \begin{align*}
        \tr(qxq y)=\lim_{\substack{n\to\infty\\ m\to \infty}}\tr(qx_nq y_m)\le \lim_{n\to\infty}\tr(qx_n^pq)^{1/p}\lim_{m\to\infty}\tr(qy_m^{p'}q)^{1/p'}= \tr(qx^pq)^{1/p}\tr(qy^{p'}q)^{1/p'}.
    \end{align*}

    Now, we check when equality occurs. Suppose that $x$ and $y$ are nonzero positive elements with spectral gap, and that the equality in \eqref{Holderlikeineq} holds. First, equality must hold in the second inequality in each part of \eqref{ineq3}. Thus, we have 
    \begin{align*}
    y^0qx^{p/2}&=qx^{p/2},& y^{p'/2}qx^0&=y^{p'/2}q.
    \end{align*}
    Multiplying $x^{-p/2}$ to the right in the first identity and $y^{-p'/2}$ to the left in the second identity gives
    \begin{equation}
        y^0qx^0=qx^0=y^{0}q.\label{supp}
    \end{equation} 
    
    The equality in \eqref{three lines} must also hold. From \eqref{three lines0} it follows that the maximum modulus $1$ of $f(w)\left[f(1)^{w}f(0)^{1-w}\right]^{-1}$ is attained at the interior point $1/p$ of $\Omega$. Thus, by the maximum modulus principle,
    $$
    f(w)=Cf(1)^{w}f(0)^{1-w}
    $$ 
    for some constant $C$. Evaluation at $w=0$ yields $C=1$, and thus 
    $$
    f(w)=f(1)^{w}f(0)^{1-w}.
    $$
    This implies that 
    $$
    |f(1+iv)|=|f(1)^{1+iv}f(0)^{-iv}|=|f(1)|=\tr(qx^pqy^0).
    $$
    Therefore, equality in \eqref{cs} holds for $u=1$ and for all $v\in\mathbb{R}$.    
    By the equality condition of the Cauchy-Schwarz inequality, we have 
    $y^0qx^{p(\frac{1}{2}+iv)}= C_v y^{ip'v}qx^{p/2}$ for some constant $C_v$ depending on $v$, or equivalently by \eqref{supp},
    \begin{equation}
        qx^{piv}= C_v y^{ip'v}q.\label{xy comm}
    \end{equation}
    Therefore,
    \begin{align*}
        C_vf(1)=\tr((C_vy^0q)x^pq)=\tr(y^{-ip'v}qx^{p(1+iv)}q)=f(1+iv)=f(1)^{1+iv}f(0)^{-iv}.
    \end{align*}
    Hence, we have $C_v=f(1)^{iv}f(0)^{-iv}$. Using \eqref{xy comm}, substituting $\widetilde{x}=\frac{x^p}{f(1)}$ and $\widetilde{y}=\frac{y^{p'}}{f(0)}$, and applying the Stone-Weierstrass theorem gives $q\widetilde{x}=\widetilde{y}q$.
    From this, we obtain $q\widetilde{x}=q^2\widetilde{x}=q\widetilde{y}q=\widetilde{x}q^2=\widetilde{x}q$.
    By continuous functional calculus, we have $[x,q]=0$. A similar argument gives $[y,q]$. Therefore, \eqref{Holderlikeineq} is equivalent to
    $$
    \tr((qx)(qy))\le \tr((qx)^p)^{1/p}\tr((qy)^{p'})^{1/p'}
    $$
    which is the noncommutative H\"older inequality. Therefore, equality holds only if $[q,x]=0=[q,y]$ and $(qx)^p=\lambda (qy)^{p'}$ for some $\lambda\in (0,\infty)$. One can easily verify that equality in \eqref{Holderlikeineq} holds for any such $x$ and $y$.
\end{proof}

\subsection{Noncommutative Minkowski's integral inequality}

\begin{prop}\label{positive}
    Let $(\M,\tr)$ be a tracial von Neumann algebra. Then $x\in \M$ is positive if and only if $\tr(rx)\ge0$ for every projection $r\in \M$.
\end{prop}
The author thanks Bisch for providing the following short proof.
\begin{proof}
    The only if part is trivial. For the if part, we decompose $x$ as $x=x_1+ix_2$ where both $x_1$ and $x_2$ are self-adjoint. We have $0\le  \tr(\mathbbm{1}_{(0,\infty)}(x_2)x_2)= \im\tr(\mathbbm{1}_{(0,\infty)}(x_2)x)=0$. Thus, $\mathbbm{1}_{(0,\infty)}(x_2)x_2=0$ by faithfulness of $\tr$. A similar argument shows $\mathbbm{1}_{(-\infty,0)}(x_2)x_2=0$. Therefore, $x_2=0$ and $x=x^*$. For a self-adjoint $x$, $\mathbbm{1}_{(-\infty,0)}(x)x\le 0$ and thus, $\tr(\mathbbm{1}_{(-\infty,0)}(x)x)\le 0$. Therefore, $\tr(\mathbbm{1}_{(-\infty,0)}(x)x)=0$ by the assumption, and $\mathbbm{1}_{(-\infty,0)}(x)x=0$ by faithfulness. Therefore, $x\ge 0$.     
\end{proof}

\begin{defn}
    Let $\N\subset \M$ be a unital inclusion of von Neumann algebras. A linear map $E:\M\rightarrow \N$ is called a \textit{conditional expectation} onto $\N$ if
    \begin{enumerate}
        \item $E(\M_+)=\N_+$;
        \item $E(x)=x$ for all $x\in \N$;
        \item $E(y_1xy_2)=y_1E(x)y_2$ for all $x\in \M$ and $y_1,y_2\in \N$.
    \end{enumerate}
\end{defn}
Let $(\M,\tr)$ be a tracial von Neumann algebra, $\N\subset \M$ be a unital subalgebra. Recall that there is a unique $\tr$-preserving conditional expectation $E:\M\rightarrow \N$ and that $\|E(x)\|\le \|x\|$ for all $x\in \M$ \cite[Section 9.1]{II1}.
\textit{Throughout, every conditional expectation is assumed to be trace preserving.}
\begin{defn}[\cite{MR999799}]
    Let $M$ be a tracial von Neumann algebra, $\N\subset \K,\Ll\subset \M$ be subalgebras, $E_{\mathcal{A}}:\M\rightarrow \mathcal{A}$ be the $\tr$-preserving conditional expectation onto a subalgebra $\mathcal{A}$. Then  
    \begin{equation*}
    \begin{tikzcd}
        \Ll\arrow[r,symbol=\subset,"E_\Ll" {yshift =3pt}] & \M\\
        \N \arrow[r,symbol=\subset,"E_\N"' {yshift=-3pt}] \arrow[u,symbol=\subset,"E_\N"{xshift=-3pt}] & \K\arrow[u,symbol=\subset,"E_\K"'{xshift=3pt}]
    \end{tikzcd}
    \end{equation*}
    is called a \textit{commuting square} if $E_\Ll E_\K=E_\N$, or equivalently, $E_\K E_\Ll =E_\N$.
\end{defn}

\begin{lem}\label{cond exp center}
    Let $\N\subset \M$ be an inclusion of von Neumann algebras and $E:\M\rightarrow \N$ be a conditional expectation. Then $E(Z(\M))\subseteq Z(\N)$.
\end{lem}

\begin{prop}[{\cite[Theorem 2.1.6.(i)]{NCLp}}]\label{dual norm}
    For every $x\in \M_+$ and $1<p<\infty$ there is $y\in \M_+$ such that $\|y\|_{p'}=1$ and $\|x\|_{p}=\tr(xy)$.
\end{prop}

\begin{thm}[Noncommutative Minkowski integral inequality]\label{minkowski}
    Let $$
    \begin{tikzcd}
        \Ll\arrow[r,symbol=\subset, "E_\Ll" {yshift=3pt}] & \M\\
        \N \arrow[r,symbol=\subset,"E_\N"'{yshift=-3pt}] \arrow[u,symbol=\subset,"E_\N"{xshift=-3pt}] & \K\arrow[u,symbol=\subset, "E_\K"' {xshift=3pt}]
    \end{tikzcd}
    $$
    be a commuting square of tracial von Neumann algebras where $E_\Ll$, $E_\K$, and $E_\N$ are the unique $\tr$-preserving conditional expectations.
    If either 
    \begin{enumerate}[(a)]
        \item $\K$ and $\Ll$ commute, i.e. $y_1y_2=y_2y_1$ for all $y_1\in \K$ and $y_2\in \Ll$, or \label{a}
        \item $x\in Z(\M)$ \label{b}
    \end{enumerate}
    then for $1\le p<\infty$ and $x\in \M_+$ we have
    \begin{equation}\label{minkowski ineq}
        E_\Ll(E_{\K}(x)^p)^{1/p}\le E_{\K}(E_\Ll(x^p)^{1/p}),
    \end{equation}
\end{thm}
\begin{proof}
    Equality holds in the case $p=1$, and this is exactly the commuting square condition. We consider the case $1<p<\infty$. In Case \eqref{a}, $\N=\K\, \cap\, \Ll$ is commutative, and in Case \eqref{b}, $E_\Ll(E_{\K}(x)^p)^{1/p}, E_{\K}(E_\Ll(x^p)^{1/p})\in Z(\N)$ by Lemma \ref{cond exp center}.
    Therefore, \eqref{minkowski ineq} is equivalent to
    $$
    \tr(r E_\Ll(E_{\K}(x)^p)^{1/p})\le \tr( r E_\K(E_{\Ll}(x^p)^{1/p}))
    $$
    for all projections $r\in Z(\N)$ by Proposition \ref{positive}.
    Hence, we need to show the above trace inequality.
    
    Suppose that $\Ll$ and $\K$ commute (resp. $x\in Z(\M)$). Without loss of generality, we may assume $\|x\|=\|x\|_\infty=1$. This implies $\|E_\Ll(E_\K(x)^p)\|$, $\|E_\Ll(x^p)\|\le 1$.  
    For $n,i,j\in \mathbb{N}$ satisfying $1\le i\le n$ and $0<t<1$, define $q_{n,i}$ and $r_{t,j}$ by
    $$
    q_{n,i}=\mathbbm{1}_{(\frac{i-1}{n},\frac{i}{n}]}\left(E_\Ll(x^p)\right)\ \ \ \text{ and } r_{t,j}= \mathbbm{1}_{(t^{j},t^{j-1}]}\left(rE_\Ll(E_\K(x)^p)r\right).
    $$
    Note that $[q_{n,i},r_{t,j}]=0$, since $[\K,\Ll]=0$ (resp. by Lemma \ref{cond exp center}).
    By the spectral theorem, we have
    $$
    E_\Ll(x^p)^{1/p}=\lim_{n\to \infty}\sum_{i=1}^n \left(\frac{i}{n}\right)^{1/p}q_{n,i}\ \ \ \text{ and }\ \ \ t^{j} r_{t,j}\le r_{t,j}E_\Ll(E_\K(x)^p)r_{t,j}\le t^{j-1} r_{t,j}.
    $$
    The above limit converges in operator norm.
    Moreover, we have 
    \begin{align}
        &\tr(rE_\Ll(E_\K(x)^p)^{1/p})=\tr\left( \sum_{j=1}^\infty r_{t,j} E_\Ll(E_\K(x)^p)^{1/p}\right) \label{infty sum}\\
        &=\sum_{j=1}^\infty\tr( r_{t,j} E_\Ll(E_\K(x)^p)^{1/p})\le \sum_{j=1}^\infty t^{-\frac{p-1}{p}j}\tr( r_{t,j} E_\Ll(E_\K(x)^p)). \label{ineq0}
    \end{align}
    The infinite sum in \eqref{infty sum} converges in norm. Thus, by the continuity of $\tr$, summation and taking the trace commute. 
    
    Now, we claim that $t^{-\frac{p-1}{p}j}\tr( r_{t,j} E_\Ll(E_\K(x)^p))\le t^{-\frac{p-1}{p}}
    \tr(r_{t,j}E_\Ll(x^p)^{1/p})$ for all $j$. We first bound $\tr(r_{t,j}E_\Ll(E_\K(x)^p))$. Since $r_{t,j}\in \N=\K\cap \Ll$ and $\sum_{i=1}^n q_{n,i}=1$,
    \begin{align}
        &\tr(r_{t,j}E_\Ll(E_\K(x)^p))=\tr(E_\Ll(r_{t,j}E_\K(x)^p))=\tr(r_{t,j}E_\K(x)E_\K(x)^{p-1})\nonumber\\
        &=\tr(E_\K(r_{t,j}xE_\K(x)^{p-1}))=\tr(r_{t,j}xE_\K(x)^{p-1})=\sum_{i=1}^n\tr(r_{t,j}q_{n,i}xE_\K(x)^{p-1})\nonumber\\
        &=\sum_{i=1}^n\tr((r_{t,j}q_{n,i}) x(r_{t,j}q_{n,i})E_\K(x)^{p-1}). \label{ineq1}
    \end{align}
    In the last identity, we used traciality and $[E_\K(x)^{p-1},q_{n,i}]=[E_\K(x)^{p-1},r_{t,j}]=[q_{n,i},r_{t,j}]=0$ (resp. $x\in Z(\M)$ and $q_{n,i}\in Z(\Ll)$).
    Now we have
    \begin{align*}
        &\tr((r_{t,j}q_{n,i}) x(r_{t,j}q_{n,i})E_\K(x)^{p-1})\le \tr(r_{t,j}q_{n,i}x^p)^{\frac{1}{p}}\tr(r_{t,j}q_{n,i}E_\K(x)^p)^{\frac{p-1}{p}}\\
        &=\tr(r_{t,j}q_{n,i}E_\Ll(x^p))^{\frac{1}{p}}\tr(r_{t,j}q_{n,i}E_\Ll(E_\K(x)^p))^{\frac{p-1}{p}}\\
        &\le\left(\frac{i}{n}\right)^{\frac{1}{p}}\tr(r_{t,j}q_{n,i})^{\frac{1}{p}}\cdot t^{\frac{p-1}{p}(j-1)}\tr(r_{t,j}q_{n,i})^{\frac{p-1}{p}}=\left(\frac{i}{n}\right)^{\frac{1}{p}}t^{\frac{p-1}{p}(j-1)}\tr(r_{t,j}q_{n,i}).
    \end{align*}
    The first inequality above follows from Lemma \ref{Holderlike}. The first identity follows from the fact that $r_{t,j},q_{n,i}\in \Ll$. From the above inequality, \eqref{ineq1} and the spectral theorem, we obtain
    \begin{align*}
        t^{-\frac{p-1}{p}j}\tr(r_{t,j}E_\Ll(E_\K(x)^p))\le t^{-\frac{p-1}{p}}\tr\left(r_{t,j}\sum_{i=1}^n \left(\frac{i}{n}\right)^{1/p} q_{n,i}\right)\xrightarrow{n\to\infty} t^{-\frac{p-1}{p}}\tr\left(r_{t,j}E_\Ll(x^p)^{1/p}\right).
    \end{align*}
    The limit above converges in operator norm.
    
    Now from \eqref{ineq0} and the above, we have
    \begin{align*}
        &\tr\left(rE_\Ll(E_\K(x)^p)^{\frac{1}{p}}\right)\le \sum_{j=1}^\infty t^{-\frac{p-1}{p}j}\tr( r_{t,j} E_\Ll(E_\K(x)^p))\le \sum_{j=1}^\infty t^{-\frac{p-1}{p}} \tr\left( r_{t,j}E_\Ll(x^p)^{\frac{1}{p} }\right)\\
        &=t^{-\frac{p-1}{p}} \tr\left(\sum_{j=1}^\infty r_{t,j}E_\Ll(x^p)^{\frac{1}{p} }\right)\le t^{-\frac{p-1}{p}} \tr\left(rE_\Ll(x^p)^{\frac{1}{p} }\right)\xrightarrow{t\to 1^-}\tr\left(rE_\Ll(x^p)^{\frac{1}{p}}\right).
    \end{align*}
    We used normality of the trace in the first identity and $\sum_{j=1}^\infty r_{t,j}\le r$ in the last inequality. 
\end{proof}

\begin{rmk}
    The proof of Theorem \ref{minkowski} becomes easier for finite-dimensional von Neumann algebras, i.e. multi-matrix algebras. One can replace $q_{n,i}$ and $r_{t,j}$ by the spectral projections $q_i$ and $r_j$ of $E_\Ll(x^p)$ and $E_\Ll(E_\K(x)^p)$, respectively. One can further assume $q_i$ and $r_j$ are minimal by decomposing projections into minimal projections. 
    This allows us to skip the steps involving limits.
\end{rmk}

\begin{ex}
There are natural examples that satisfy Condition \eqref{a} in Theorem \ref{minkowski}.
\begin{enumerate}
    \item Let $\Ll\subset \M$ be an inclusion of tracial von Neumann algebras, $\K=\Ll'\cap \M$, and $\N=Z(\Ll)$. In particular, any $\rm{II}_1$ subfactor $\Ll\subset \M$ can be extended to such a commuting square.
    \item Let $X$ and $Y$ be objects in a spherical unitary tensor category $\mathcal{C}$. Set $\M=\End(X\otimes Y)$, $\K=\End(X)\otimes \id_Y$, $\Ll=\id_X\otimes \End(Y)$, and $\N=\id_{X}\otimes \End(1_{\mathcal{C}})\otimes \id_{Y}$.
\end{enumerate}
\end{ex}

\begin{cor}
    Let $\N\subset \Ll,\K\subset \M$ and $x\in \M_+$ be as in Theorem \ref{minkowski}. Then for $1\le p_1\le p_2<\infty$ we have 
    $$
    E_\Ll(E_\K(x^{p_1})^{p_2/p_1})^{1/p_2} \le E_\K(E_\Ll(x^{p_2})^{p_1/p_2})^{1/p_1}.
    $$
\end{cor}
\begin{proof}
    Replace $x$ and $p$ in Theorem \ref{minkowski} by $x_1^{p_1}$ and $p_2/p_1$, respectively. Taking the power with $1/p_1$ on both sides preserves the inequality by the L\"owner-Heinz theorem \cite{MR1545446,MR44747, MR53390, MR306957}.
\end{proof}

\section{Applications}\label{Sec_App}
\subsection{Construction of commuting squares}
A \textit{multi-matrix algebra} is an algebra of the form $\bigoplus_{i=1}^n M_{k_i}(\mathbb{C})$. All finite-dimensional von Neumann algebras are multi-matrix algebras. Consider an inclusion 
$$
A_0=\bigoplus_{i=1}^{n_0} M_{k_i}(\mathbb{C})\subset \bigoplus_{j=1}^{n_1} M_{l_j}(\mathbb{C})=A_1
$$
of multi-matrix algebras. Let $m_{ij}$ be the multiplicity of the summand $M_{k_i}(\mathbb{C})$ in $M_{l_j}(\mathbb{C})$. Then the matrix $T=(m_{ij})_{i,j}$ is called an \textit{inclusion matrix}. Let $\tr:A_1\rightarrow \mathbb{C}$ be a faithful tracial state on $A_1$ and let $v_j=\tr(q_j)$ where $q_j$ is a minimal projection in the direct summand $M_{l_j}(\mathbb{C})$ of $A_1$. Then the (column) vector $v=(v_j)_j$ is called the \textit{trace vector} corresponding to $\tr$. It follows that $Tv$ is the trace vector corresponding to the induced trace $\tr|_{A_0}$ of $A_0$.

\begin{lem}[{\cite[Proposition 5.4.3]{MR1473221}}]\label{cond exp formula}
    Let $A_0=\bigoplus_{i=1}^{n_0} M_{k_i}(\mathbb{C}) \stackrel{T}{\subset} \bigoplus_{j=1}^{n_1} M_{l_j}(\mathbb{C})= A_1$ be an inclusion of multi-matrix algebras, $T=(m_{ij})_{i,j}$ be the inclusion matrix, $z_i^0$ and $z_j^1$ be the central projections corresponding to the $i$-th and $j$-th summand of $A_0$ and $A_1$, respectively, $\tr$ be a faithful tracial state on $A_1$, and $E:A_1\rightarrow A_0$ be the trace-preserving conditional expectation.  Then we have
    $$
    E(z_j^1)=\sum_{i}\frac{m_{ij}v_j^1}{v_i^0}z_i^0
    $$
    where $v^0=(v_i^0)_i$ and $v^1=(v_j^1)_j$ are the trace vectors corresponding to $\tr|_{A_0}$ and $\tr$, respectively. 
\end{lem}

\begin{thm}\label{minkowski criterion cs}
    Let 
    \begin{equation}
        \begin{tikzcd}
            A_{01}\arrow[r,symbol=\subset,"T_{01}^{11}"{yshift = 5pt}]&A_{11}\\
            A_{00}\arrow[r,symbol=\subset,"T_{00}^{10}"'{yshift = -5pt}] \arrow[u,symbol=\subset,"T_{00}^{01}" {xshift = -5pt}]&A_{10}\arrow[u,symbol=\subset,"T_{10}^{11}"'{xshift = 5pt}]
        \end{tikzcd}
    \end{equation}
    be a commuting square of multi-matrix algebras. Here $T_{ij}^{kl}$ is the inclusion matrix for $A_{ij}\subset A_{kl}$.
    Let $\tr:A_{11}\rightarrow \mathbb{C}$ be a trace on $A_{11}$ with the trace vector $v^{11}=(v_k^{11})_{k=1}^{n_{11}}$ and let $v^{ij}=(v_{k}^{ij})_{k=1}^{n_{ij}}$ be the induced trace vector of $A_{ij}$. Then for all $1\le p<\infty$, $(a_l)_{l=1}^{n_{11}}$ with $a_l\ge 0$ and $1\le i\le n_{00}$, we have
    \begin{equation}
        \left(\sum_{j=1}^{n_{10}}\left(\sum_{l=1}^{n_{11}} \frac{(T_{10}^{11})_{jl}\ a_lv_l^{11}}{v_j^{10}}\right)^p \frac{(T_{00}^{10})_{ij}\ v_j^{10}}{v_i^{00}}\right)^{\frac{1}{p}} \le \sum_{k=1}^{n_{01}}\left(\sum_{l=1}^{n_{11}}  \frac{(T_{01}^{11})_{kl}\ a_l^pv_l^{11}}{v_k^{01}}\right)^{\frac{1}{p}}\frac{(T_{00}^{01})_{ik}\ v_k^{01}}{v_i^{00}}. \label{minkowski criterion cs ineq}
    \end{equation}
\end{thm}
\begin{proof}
    Let $E_{ij}:A_{11}\rightarrow A_{ij}$ be the $\tr$-preserving conditional expectations onto $A_{ij}$ and $z_k^{ij}$ be the central projection of $A_{ij}$ corresponding to the $k$-th summand.
    By Theorem \ref{minkowski} we have
    $$
    E_{01}\left(E_{10}\left(\sum_{l=1}^{n_{11}} a_lz_l^{11}\right)^p\right)^{\frac{1}{p}}\le E_{10}\left(E_{01}\left(\sum_{l=1}^{n_{11}} a_l^pz_l^{11}\right)^{\frac{1}{p}}\right).
    $$
    We expand the conditional expectation using Lemma \ref{cond exp formula}:
    \begin{align*}
        &E_{01}\left(E_{10}\left(\sum_{l=1}^{n_{11}} a_lz_l^{11}\right)^p\right)^{\frac{1}{p}}=E_{01}\left(\left(\sum_{l=1}^{n_{11}} a_l\sum_{j=1}^{n_{10}}\frac{(T_{10}^{11})_{jl}\ v_l^{11}}{v_j^{10}}z_j^{10}\right)^p\right)^{\frac{1}{p}}\\
        &=E_{01}\left(\sum_{j=1}^{n_{10}}\left(\sum_{l=1}^{n_{11}} \frac{(T_{10}^{11})_{jl}\ a_lv_l^{11}}{v_j^{10}}\right)^p z_j^{10}\right)^{\frac{1}{p}}
        =\left(\sum_{j=1}^{n_{10}}\left(\sum_{l=1}^{n_{11}} \frac{(T_{10}^{11})_{jl}\ a_lv_l^{11}}{v_j^{10}}\right)^p \sum_{i=1}^{n_{00}}\frac{(T_{00}^{10})_{ij}\ v_j^{10}}{v_i^{00}}z_i^{00}\right)^{\frac{1}{p}}\\
        &=\sum_{i=1}^{n_{00}}\left(\sum_{j=1}^{n_{10}}\left(\sum_{l=1}^{n_{11}} \frac{(T_{10}^{11})_{jl}\ a_lv_l^{11}}{v_j^{10}}\right)^p \frac{(T_{00}^{10})_{ij}\ v_j^{10}}{v_i^{00}}\right)^{\frac{1}{p}}z_i^{00}.\\
        &\\
        &E_{10}\left(E_{01}\left(\sum_{l=1}^{n_{11}} a_l^pz_l^{11}\right)^{\frac{1}{p}}\right)=
        E_{10}\left(\left(\sum_{l=1}^{n_{11}} a_l^p \sum_{k=1}^{n_{01}}\frac{(T_{01}^{11})_{kl}\ v_l^{11}}{v_k^{01}}z_k^{01}\right)^{\frac{1}{p}}\right)\\
        &=E_{10}\left(\sum_{k=1}^{n_{01}}\left(\sum_{l=1}^{n_{11}}  \frac{(T_{01}^{11})_{kl}\ a_l^pv_l^{11}}{v_k^{01}}\right)^{\frac{1}{p}}z_k^{01}\right)=\sum_{k=1}^{n_{01}}\left(\sum_{l=1}^{n_{11}}  \frac{(T_{01}^{11})_{kl}\ a_l^pv_l^{11}}{v_k^{01}}\right)^{\frac{1}{p}}\sum_{i=1}^{n_{00}}\frac{(T_{00}^{01})_{ik}\ v_k^{01}}{v_i^{00}}z_i^{00}\\
        &=\sum_{i=1}^{n_{00}}\left(\sum_{k=1}^{n_{01}}\left(\sum_{l=1}^{n_{11}}  \frac{(T_{01}^{11})_{kl}\ a_l^pv_l^{11}}{v_k^{01}}\right)^{\frac{1}{p}}\frac{(T_{00}^{01})_{ik}\ v_k^{01}}{v_i^{00}}\right)z_i^{00}.
    \end{align*}
    Comparing the coefficients of $z_i^{00}$ gives the desired inequality.
\end{proof}

\begin{rmk}
    \begin{enumerate}
        \item Switching the role of $A_{10}$ and $A_{01}$ in Theorem \ref{minkowski criterion cs} gives an additional obstruction for existence of commuting square. This can simply obtained by switching the indices $01$ and $10$ in \eqref{minkowski criterion cs ineq}.
        \item The \textit{basic construction} of a \textit{nondegenerate} commuting square is also a nondegenerate commuting square, and their inclusion matrices are obtained by taking transpose and reordering the original inclusion matrices \cite{MR1473221}. Therefore, to see whether $(T_{00}^{01},T_{01}^{11},T_{00}^{10},T_{10}^{11})$ fails to constitute inclusion graphs of a nondegenerate commuting square, we can also test \eqref{minkowski criterion cs ineq} on 
        \begin{align*}
            &\left((T_{00}^{01})^t,T_{00}^{10},T_{01}^{11},(T_{10}^{11})^t\right),& 
            &\left(T_{10}^{11},(T_{01}^{11})^t,(T_{00}^{10})^t,T_{00}^{01}\right),&  
            &\left((T_{10}^{11})^t,(T_{00}^{10})^t,(T_{01}^{11})^t,(T_{00}^{01})^t\right).
        \end{align*}
    \end{enumerate}
\end{rmk}

\subsection{Unitary categorification of fusion rings}
A \textit{unitary tensor category} is an idempotent complete rigid C*-tensor category. Any object in a unitary tensor category has a balanced dual that is unique up to unitary isomorphism \cite[Lemma 3.9]{MR2091457}.
These choices of duals assemble into a dual functor whose associated unitary pivotal structure is spherical  \cite{MR4133163}. 
Hence, each object has positive quantum dimension.

\begin{cor}\label{minkowski criterion cat}
    Let $\mathcal{C}$ be a spherical unitary (multi-)tensor category with $s$ isomorphism classes of simple objects where $s\in \mathbb{N}\cup\{\infty\}$. Suppose that the simple objects $c_i\in \Irr(\mathcal{C})$ satisfy the fusion rule $c_i\otimes c_j\cong \bigoplus_{l=1}^s N_{ij}^l c_l$ and let $d_i$ be the quantum dimension of $c_i$. Then for any triple of simple objects $c_i,c_j,c_k\in\Irr(\mathcal{C})$, $1\le p<\infty$, and $a_n\ge0$ for $1\le n\le s$, we have
    \begin{equation}
        \left(\sum_{l=1}^{s}\left(\sum_{n=1}^{s} \frac{N_{lk}^n\ a_nd_n}{d_l}\right)^p \frac{N_{ij}^ld_l}{d_j}\right)^{\frac{1}{p}} \le \sum_{m=1}^{s}\left(\sum_{n=1}^{s}  \frac{N_{im}^n\ a_n^pd_n}{d_m}\right)^{\frac{1}{p}}\frac{N_{jk}^m d_m}{d_j}.\label{minkowski criterion cat ineq}
    \end{equation}
\end{cor}
\begin{proof}
    The quadrilateral of algebras
    \begin{equation}
        \begin{tikzcd}
            \id_{c_i}\otimes\End(c_j\otimes c_k)\arrow[r,symbol=\subset]&\End(c_i\otimes c_j\otimes c_k)\\
            \mathbb{C}\id_{c_i\otimes c_j\otimes c_k}=\id_{c_i}\otimes \End(c_j)\otimes \id_{c_k}\arrow[r,symbol=\subset] \arrow[u,symbol=\subset]&\End(c_i\otimes c_j)\otimes\id_{c_k}\arrow[u,symbol=\subset]
        \end{tikzcd}
    \end{equation}
    is a commuting square of multi-matrix algebras where the conditional expectations are the normalized partial traces, and hence, we can apply Theorem \ref{minkowski criterion cs}.
    Replacing $(T_{10}^{11})_{**}$, $(T_{00}^{10})_{1*}$, $(T_{01}^{11})_{**}$ and  $(T_{00}^{01})_{1*}$ in \eqref{minkowski criterion cs ineq} by $N_{*k}^*$, $N_{ij}^*$, $N_{i*}^*$ and $N_{jk}^*$, and replacing $v_*^{11}$, $v_*^{10}$, $v_*^{01}$ and $v_*^{00}$ by $\frac{d_*}{d_id_jd_k}$, $\frac{d_*}{d_id_j}$, $\frac{d_*}{d_jd_k}$ and $1$, respectively, gives
    $$
    \left(\sum_{l=1}^{s}\left(\sum_{n=1}^{s} \frac{N_{lk}^n\ a_nd_n}{d_ld_k}\right)^p \frac{N_{ij}^ld_l}{d_id_j}\right)^{\frac{1}{p}} \le \sum_{m=1}^{s}\left(\sum_{n=1}^{s}  \frac{N_{im}^n\ a_n^pd_n}{d_id_m}\right)^{\frac{1}{p}}\frac{N_{jk}^m d_m}{d_jd_k}.
    $$
    Multiplying both sides by $d_i^{\frac{1}{p}}d_k$ results in the desired inequality.
\end{proof}

\begin{defn}[\cite{MR933415,MR3242743}]
    \begin{enumerate}
        \item A \textit{unital based ring} $\mathcal{R}$ is a $\mathbb{Z}$-algebra that is free as a $\mathbb{Z}$-module, endowed with a fixed basis $\{b_i\}_{i\in I}$ with the following properties:
        \begin{enumerate}
            \item $b_ib_j=\sum_{k\in I}N_{ij}^k b_k$ for some $N_{ij}^k\in \mathbb{Z}_{\ge0}$ where $N_{ij}^k=0$ for all but finitely many $k$;
            \item There is a designated element $1\in I$ such that $b_1=1$;
            \item There exists an involution $i\mapsto \overline{i}$ of $I$ such that the induced map $\sum_{i\in I}n_ib_i\mapsto \sum_{i\in I}n_ib_{\overline{i}}$ is an anti-isomorphism of rings; 
            \item For every $i,j\in I$, $N_{ij}^1=\delta_{j,\overline{i}}$.
        \end{enumerate}
        \item A unital based ring with a finite basis is called a \textit{fusion ring}.
    \end{enumerate}
\end{defn}

Grothendieck rings of fusion categories are fusion rings. 
Corollary \ref{minkowski criterion cat} can be used to disqualify unital based rings from having unitary categorifications: If there is a triple of basis elements $c_i,c_j,c_k$ of a unital based ring $\mathcal{R}$, nonnegative numbers $(a_n)_n$ and $1\le p<\infty$ that do not satisfy Corollary \ref{minkowski criterion cat}, then $\mathcal{R}$ does not admit a unitary categorification. In particular, this criterion can be applied to fusion rings. 
However, difficulties lie in choosing the appropriate $(a_n)_n$ and $p$. We applied the following algorithm to each triple $(c_i,c_j,c_k)$ for our calculations:
\begin{enumerate}
    \item Randomly choose an element $(a_n)_n\in [0,1]^s$ and $1/p\in (0.1,1]$.
    \item Apply the \textit{gradient descent} method to $\RHS/\LHS$ where $\RHS$ and $\LHS$ are the right- and left-hand sides of \eqref{minkowski criterion cat ineq}.
    \item If $\RHS/\LHS<1$, then exclude the fusion ring. 
\end{enumerate}
The numbers $a_n$ and $1/p$ are sampled from the uniform distribution.
To approximate the partial derivative of $\RHS/\LHS$ with respect to $a_n$ and $1/p$, we computed the difference quotient with $\Delta a_n=\Delta (1/p)=0.0001$. The gradient descent method is applied with step size $0.05$ and $1,000$ iterations. The overall process is implemented in MATLAB and is repeated up to $20$ times for each choice of a fusion ring. 

We tested the algorithm on the dataset of $28,450$ fusion rings from \cite{MR4644689} and excluded $5,793$ of them (20.36\%). Among them, $540$ fusion rings ($1.898\%$) cannot be excluded by the \textit{primary $3$-criterion} \cite{MR4686667}. Moreover, $23$ of these cannot be excluded by the following criteria in \cite{MR4434141}: \textit{$d$-number}, \textit{extended cyclotomic}, \textit{Lagrange}, and \textit{pseudo-unitarity Drinfeld}. Some of these admit \textit{induction matrices} and thus cannot be excluded by the criterion in \cite[Remark 3.2]{ABDP25}. For more details of the results, see Tables \ref{excluded} and \ref{exclusion rate}. The author thanks Palcoux for testing the criteria in \cite{MR4434141} and \cite{ABDP25}.

Due to the random nature of the algorithm, some fusion rings are not excluded by our criterion, even though they should.\footnote{We confirmed this by running the algorithm multiple times for certain subsets of fusion rings. At a certain trial run, the number of excluded fusion rings was higher than what is provided in Table \ref{excluded}. However, we lost the data and do not know which additional fusion rings were excluded.} These fusion rings can be excluded through more iterations and by adjusting the step size.

\begin{ex}
The following fusion rings pass the primary $3$-criterion, but fail to satisfy \eqref{minkowski criterion cat ineq}. Here, the $(n,m)$-entry of the $l$-th matrix is $N_{lm}^n$. 
\begin{enumerate}
    \item The following is the simplest such a fusion ring:
    $$
\begin{array}{cccc}
	 \begin{bmatrix} 
		 1 & 0& 0& 0\\
		 0 & 1& 0& 0\\
		 0 & 0& 1& 0\\
		 0 & 0& 0& 1
	\end{bmatrix} 
 & 	 \begin{bmatrix} 
		 0 & 1& 0& 0\\
		 1 & 1& 1& 1\\
		 0 & 1& 0& 0\\
		 0 & 1& 0& 0
	\end{bmatrix} 
 & 	 \begin{bmatrix} 
		 0 & 0& 0& 1\\
		 0 & 1& 0& 0\\
		 1 & 0& 0& 0\\
		 0 & 0& 1& 0
	\end{bmatrix} 
 & 	 \begin{bmatrix} 
		 0 & 0& 1& 0\\
		 0 & 1& 0& 0\\
		 0 & 0& 0& 1\\
		 1 & 0& 0& 0
	\end{bmatrix} 
\end{array}
$$
For $i=j=k=2$, $\frac{1}{p}=0.661975038859587$, and 
\begin{align*}
    a=(a_n)_{n=1}^4&=(0.778280214647139,
    0.006040665900893,
    0.698943990142325,
    0.856127138742457)
\end{align*}
we have $\RHS/\LHS=0.923133094619110<1$.
    \item The following is the simplest integral example among those that cannot be even ruled out by the criteria in \cite{MR4434141} and \cite{ABDP25}: 
    {\footnotesize
    \begin{align*}
        \begin{array}{ccccccc}
	 \begin{bmatrix} 
		 1 & 0& 0& 0& 0& 0& 0\\
		 0 & 1& 0& 0& 0& 0& 0\\
		 0 & 0& 1& 0& 0& 0& 0\\
		 0 & 0& 0& 1& 0& 0& 0\\
		 0 & 0& 0& 0& 1& 0& 0\\
		 0 & 0& 0& 0& 0& 1& 0\\
		 0 & 0& 0& 0& 0& 0& 1
	\end{bmatrix} 
 & 	 \begin{bmatrix} 
		 0 & 1& 0& 0& 0& 0& 0\\
		 1 & 0& 0& 0& 0& 0& 0\\
		 0 & 0& 0& 1& 0& 0& 0\\
		 0 & 0& 1& 0& 0& 0& 0\\
		 0 & 0& 0& 0& 1& 0& 0\\
		 0 & 0& 0& 0& 0& 1& 0\\
		 0 & 0& 0& 0& 0& 0& 1
	\end{bmatrix} 
 & 	 \begin{bmatrix} 
		 0 & 0& 1& 0& 0& 0& 0\\
		 0 & 0& 0& 1& 0& 0& 0\\
		 1 & 0& 0& 0& 0& 0& 0\\
		 0 & 1& 0& 0& 0& 0& 0\\
		 0 & 0& 0& 0& 1& 0& 0\\
		 0 & 0& 0& 0& 0& 1& 0\\
		 0 & 0& 0& 0& 0& 0& 1
	\end{bmatrix} 
 & 	 \begin{bmatrix} 
		 0 & 0& 0& 1& 0& 0& 0\\
		 0 & 0& 1& 0& 0& 0& 0\\
		 0 & 1& 0& 0& 0& 0& 0\\
		 1 & 0& 0& 0& 0& 0& 0\\
		 0 & 0& 0& 0& 1& 0& 0\\
		 0 & 0& 0& 0& 0& 1& 0\\
		 0 & 0& 0& 0& 0& 0& 1
	\end{bmatrix} \\
 	 \begin{bmatrix} 
		 0 & 0& 0& 0& 1& 0& 0\\
		 0 & 0& 0& 0& 1& 0& 0\\
		 0 & 0& 0& 0& 1& 0& 0\\
		 0 & 0& 0& 0& 1& 0& 0\\
		 1 & 1& 1& 1& 0& 0& 0\\
		 0 & 0& 0& 0& 0& 2& 0\\
		 0 & 0& 0& 0& 0& 0& 2
	\end{bmatrix} 
 & 	 \begin{bmatrix} 
		 0 & 0& 0& 0& 0& 0& 1\\
		 0 & 0& 0& 0& 0& 0& 1\\
		 0 & 0& 0& 0& 0& 0& 1\\
		 0 & 0& 0& 0& 0& 0& 1\\
		 0 & 0& 0& 0& 0& 0& 2\\
		 1 & 1& 1& 1& 2& 1& 1\\
		 0 & 0& 0& 0& 0& 3& 1
	\end{bmatrix} 
 & 	 \begin{bmatrix} 
		 0 & 0& 0& 0& 0& 1& 0\\
		 0 & 0& 0& 0& 0& 1& 0\\
		 0 & 0& 0& 0& 0& 1& 0\\
		 0 & 0& 0& 0& 0& 1& 0\\
		 0 & 0& 0& 0& 0& 2& 0\\
		 0 & 0& 0& 0& 0& 1& 3\\
		 1 & 1& 1& 1& 2& 1& 1
	\end{bmatrix} 
\end{array}
    \end{align*}
    }
    For $i=7$, $j=6$, $k=7$, $\frac{1}{p}=0.393251890984615$ and
    \begin{align*}
        a=(&0.913028612500332,   0.551159196648294,   0.759269438861921,   0.798984777363567, \\
        &0.526885050265131,   0.170238904317482,   0.000598136821664)   
    \end{align*}
    we have $\RHS/\LHS=0.767708161620070<1$.
\end{enumerate}
\end{ex}

\begin{center} 
\begin{sidewaystable}[]
{\small\begin{center}
\begin{tabular}{|c || c| c| c| c| c| c| c| c|| c|} 
 \hline 
\backslashbox{$\max N_{ij}^k$}{Rank} & 2 & 3 & 4 & 5 & 6 & 7 & 8 & 9 & Total \\ 
\hline 
1 & 0/2 & 0/4 & 1/10 & 4/16 & 10/39 & 19/43 & 22/96 & 51/142 & 107/352 \\ 
\hline 
2 & 0/1 & 1/3 & 3/17 & 15/37 & 56/154 & 160/319 & 397/874 & - & 632/1405 \\ 
\hline 
3 & 0/1 & 1/4 & 5/24 & 20/82 & 140/384 & 217/562 & - & - & 383/1057 \\ 
\hline 
4 & 0/1 & 2/6 & 10/45 & 34/134 & 268/872 & 374/1236 & - & - & 688/2294 \\ 
\hline 
5 & 0/1 & 1/5 & 8/55 & 35/209 & 145/533 & - & - & - & 189/803 \\ 
\hline 
6 & 0/1 & 1/9 & 17/81 & 71/336 & 201/872 & - & - & - & 290/1299 \\ 
\hline 
7 & 0/1 & 1/6 & 14/92 & 88/477 & 202/976 & - & - & - & 305/1552 \\ 
\hline 
8 & 0/1 & 1/10 & 25/137 & 129/733 & 367/1672 & - & - & - & 522/2553 \\ 
\hline 
9 & 0/1 & 2/12 & 28/151 & 230/1463 & - & - & - & - & 260/1627 \\ 
\hline 
10 & 0/1 & 1/9 & 34/186 & 302/1794 & - & - & - & - & 337/1990 \\ 
\hline 
11 & 0/1 & 2/10 & 43/238 & 334/2283 & - & - & - & - & 379/2532 \\ 
\hline 
12 & 0/1 & 1/20 & 61/291 & 479/3049 & - & - & - & - & 541/3361 \\ 
\hline 
13 & 0/1 & 1/9 & 41/246 & 208/1300 & - & - & - & - & 250/1556 \\ 
\hline 
14 & 0/1 & 2/13 & 47/340 & 216/1323 & - & - & - & - & 265/1677 \\ 
\hline 
15 & 0/1 & 1/16 & 58/349 & 223/1550 & - & - & - & - & 282/1916 \\ 
\hline 
16 & 0/1 & 3/25 & 74/525 & 286/1925 & - & - & - & - & 363/2476 \\ 
\hline 
\hline 
Total & 0/17 & 21/161 & 469/2787 & 2674/16711 & 1389/5502 & 770/2160 & 419/970 & 51/142 & 5793/28450 \\ 
\hline 
\end{tabular} 
\end{center}
}
\caption{
\centering The number of excluded fusion rings out of the total number of fusion rings in the dataset \\
by the Minkowski criterion, classified by the rank and the multiplicity (excluded/total)\\
(Total running time: 97 d 15 h 22 m 12 s)\label{excluded}} 

\bigskip

{\small
\begin{center}
\begin{tabular}{|c || c| c| c| c| c| c| c| c|| c|} 
 \hline 
\backslashbox{$\max N_{ij}^k$}{Rank} & 2 & 3 & 4 & 5 & 6 & 7 & 8 & 9 & Total \\ 
\hline 
1 & 0.00 & 0.00 & 10.00 & 25.00 & 25.64 & 44.19 & 22.92 & 35.92 & 30.40 \\ 
\hline 
2 & 0.00 & 33.33 & 17.65 & 40.54 & 36.36 & 50.16 & 45.42 & - & 44.98 \\ 
\hline 
3 & 0.00 & 25.00 & 20.83 & 24.39 & 36.46 & 38.61 & - & - & 36.23 \\ 
\hline 
4 & 0.00 & 33.33 & 22.22 & 25.37 & 30.73 & 30.26 & - & - & 29.99 \\ 
\hline 
5 & 0.00 & 20.00 & 14.55 & 16.75 & 27.20 & - & - & - & 23.54 \\ 
\hline 
6 & 0.00 & 11.11 & 20.99 & 21.13 & 23.05 & - & - & - & 22.32 \\ 
\hline 
7 & 0.00 & 16.67 & 15.22 & 18.45 & 20.70 & - & - & - & 19.65 \\ 
\hline 
8 & 0.00 & 10.00 & 18.25 & 17.60 & 21.95 & - & - & - & 20.45 \\ 
\hline 
9 & 0.00 & 16.67 & 18.54 & 15.72 & - & - & - & - & 15.98 \\ 
\hline 
10 & 0.00 & 11.11 & 18.28 & 16.83 & - & - & - & - & 16.93 \\ 
\hline 
11 & 0.00 & 20.00 & 18.07 & 14.63 & - & - & - & - & 14.97 \\ 
\hline 
12 & 0.00 & 5.00 & 20.96 & 15.71 & - & - & - & - & 16.10 \\ 
\hline 
13 & 0.00 & 11.11 & 16.67 & 16.00 & - & - & - & - & 16.07 \\ 
\hline 
14 & 0.00 & 15.38 & 13.82 & 16.33 & - & - & - & - & 15.80 \\ 
\hline 
15 & 0.00 & 6.25 & 16.62 & 14.39 & - & - & - & - & 14.72 \\ 
\hline 
16 & 0.00 & 12.00 & 14.10 & 14.86 & - & - & - & - & 14.66 \\ 
\hline 
\hline 
Total & 0.00 & 13.04 & 16.83 & 16.00 & 25.25 & 35.65 & 43.20 & 35.92 & 20.36 \\ 
\hline 
\end{tabular} 
\end{center}
}
\caption{The exclusion rate of the Minkowski criterion (\%)\label{exclusion rate}}
\end{sidewaystable} 
\end{center}

\FloatBarrier
\bibliographystyle{alpha} 
\bibliography{References}{}

\end{document}